\documentclass[12pt]{amsart}
\usepackage{amsmath}
\usepackage{amsxtra}
\usepackage{amstext}
\usepackage{amssymb}
\usepackage{amsthm}
 \usepackage{amsfonts}
\usepackage{latexsym}
\usepackage{url}
\usepackage{dsfont} 
\usepackage{verbatim}
\usepackage{tabls}
\usepackage{graphicx}
\usepackage{rotating}
\usepackage{caption}
\usepackage[shortlabels]{enumitem}

\usepackage[pagebackref,hypertexnames=false, colorlinks, citecolor=red, linkcolor=red]{hyperref}
\usepackage[backrefs]{amsrefs}

\usepackage[inline,nomargin]{fixme}

\theoremstyle{plain}

\newtheorem{thm}{Theorem}[section]

\newtheorem{lem}[thm]{Lemma}
\newtheorem{prop}[thm]{Proposition}

\newtheorem{fact}[thm]{Fact}

\theoremstyle{definition}

\newtheorem{cla}[thm]{Claim}

\numberwithin{equation}{section}
\def\R1{\widetilde{R}}
\def\T1{\widetilde{T}}

\def\supp{\operatorname{supp}}

\def\eps{\varepsilon}
\def\kap{\varkappa}
\def\R{\mathbb{R}}

\def\diam{\operatorname{diam}}

\def\wh{\widehat}
\def\wt{\widetilde}

\def\S{\mathbb{S}}

\def\kap{\varkappa}

\def\T{\mathbb{T}}
\def\H{\mathcal{H}}
\def\PW{\mathcal{P}\mathcal{W}}
\def\B{\mathcal{P}\mathcal{W}_{\infty}}
\def\N{\mathbb{N}}
\def\C{\mathbb{C}}
\def\W{\mathbf{W}}
\def\D{\mathbf{D}}
\def\Beu{\textbf{D}_{\text{Beu}}}
\def\card{\operatorname{card}}

\def\XXint#1#2#3{{\setbox0=\hbox{$#1{#2#3}{\int}$}
     \vcenter{\hbox{$#2#3$}}\kern-.5\wd0}}

\begin{document}

\title[Mobile Sampling]{A sufficient condition for Mobile Sampling in terms of surface density}
\author{Benjamin Jaye}
\email{bjaye3@gatech.edu}
\address{School of Mathematics, Georgia Tech}
\author{Mishko Mitkovski}
\email{mmitkov@clemson.edu}
\address{School of Mathematical and Statistical Sciences, Clemson University}

\thanks{Research supported in part by NSF DMS-2049477 and DMS-2103534 to B.J and DMS-2000236 to M.M.  This work was completed while the first author was in residence at the Hausdorff Institute in Bonn as part of the trimester The Interplay between High-Dimensional Geometry and Probability.}

\maketitle

\begin{abstract}
We provide a sufficient condition for sets of mobile sampling in terms of the surface density of the set.
\end{abstract}

\section{Introduction}

One of the most fundamental problems in signal and data processing is the problem of stable recovery of a band-limited function from a set of incomplete measurements. It is well-known that a stable recovery is possible if and only if measurements are available on a set $\Gamma\subseteq \R^d$ satisfying the so called sampling inequality: There exists $c>0$ such that 
\[c\Bigl(\int_{\R^d}|f|^2 dm_d\Bigl)^{1/2} \!\!\leq \Bigl(\sum_{\gamma\in\Gamma}|f(\gamma)|^2 \,\Bigl)^{1/2}
\!\!\text{ for every }f\in \PW_2(K).\]
Here and elsewhere, we denote by $\PW_2(K)$ the classical Paley-Wiener space consisting of square-integrable functions with Fourier spectrum\footnote{We normalize the Fourier transform so that for $f\in \mathcal{S}(\R^d)$, $\wh{f}(\xi) = \int_{\R^d}f(x)e^{-2\pi i x\cdot \xi}dm_d(x)$.} contained in a compact set $K\subseteq \R^d$, and by $m_d$ the $d$-dimensional Lebesgue measure.
Sets $\Gamma$  satisfying the sampling inequality are typically called (stable) \textit{sampling sets}.  When $d=1$ and $K$ is an interval,  
Beurling \cite{Beu} and Kahane \cite{Kah}  essentially characterized sampling sets in terms of what is now known as the lower Beurling density 
$$\Beu^-(\Gamma) = \liminf_{r\to \infty}\inf_{x\in \R^d}\frac{\text{card}(\Gamma\cap B(x,r))}{m_d(B(x,r))}.$$ 
Namely, for a uniformly discrete set $\Gamma$ to be sampling for $\PW_2(-\frac{1}{2},\frac{1}{2})$ it is necessary $\Beu^-(\Gamma)\geq 1$ and it is sufficient $\Beu^-(\Gamma)> 1$.  Landau \cite{Lan} extended the necessary condition $\Beu^-(\Gamma)\geq 1$ to all dimensions $d$ and all compact spectra $K$ with Lebesgue measure $1$. Simple counter-examples show that a general sufficiency result of this type in higher dimensions is not possible.

A more general form of the sampling problem asks for a description of so called \textit{sampling measures} $\mu$ which satisfy the following appropriate analog of the sampling inequality:
There exists $c>0$ such that 
\[c\Bigl(\int_{\R^d}|f|^2 dm_d\Bigl)^{1/2} \!\!\leq \Bigl(\int_{\R^d} |f|^2 d\mu \Bigl)^{1/2}
\!\!\text{ for every }f\in \PW_2(K).\]
Interesting classes of sampling measures take the form $\mu=1_\Gamma \H^k$, where $\H^k$ is the $k$-dimensional Hausdorff measure\footnote{We normalize the Hausdorff measure so that $\H^{k}$ coincides with $m_{k}$ when restricted to a $k$-dimensional plane in $\R^d$.}. The aforementioned results of Beurling and Kahane concern the case $k=0$. At the other extreme $k=d$, the Logvinenko-Sereda \cite{LS, MS} inequality (see \cite{Pan1,Pan2, Katz} for earlier results regarding $d=1$) provides a complete description of sampling measures of the form $\mu=1_\Gamma \H^d$ when the spectrum $K$ is a $d$-dimensional ball. Surprisingly enough, until recently, very few results were available concerning the natural intermediate case $0<k<d$. 

\subsection{The mobile sampling problem} The mobile sampling problem concerns the intermediate case $\mu=1_\Gamma \H^k$ with $0<k<d$. The study of this problem in this level of generality was initiated by Unnikrishnan and Vetterli~\cite{UV12, UV13}, who formulated it precisely and coined the name mobile sampling. It should be noted that some traces of this problem already appear in the earlier work of Benedetto and Wu~\cite{BW}, who studied sampling on spiral curves in relation to MRI reconstruction. Unnikrishnan and Vetterli  characterized mobile sampling sets within a variety of special types of curves and surfaces using the concept of a \textit{path density} as an appropriate analog of the lower Beurling density. They defined a lower path density $l_k^-(\Gamma)$ of a $k$-dimensional surface $\Gamma\subseteq \R^d$ by 
\[l_k^-(\Gamma)=\liminf_{r\to \infty}\inf_{x\in \R^d}\frac{\H^{k}(\Gamma\cap B(x,r))}{m_d(B(x,r))}.\]  The mobile sampling problem has attracted a great deal of attention among mathematicians, and a variety of necessary and sufficient conditions for mobile sampling have been proved subsequently for other particular classes of surfaces where the precise shape of the support of the Fourier transform can be taken into account \cite{AGR, GRUV, NJR, RUZ}. While the vast majority of these results pertain to specific classes of surfaces, a notable general result was proved by Gr\"ochenig, Romero, Unnikrishnan, and Vetterli \cite{GRUV}, who showed that, for a given spectrum $K\subset \R^d$, the problem of finding a mobile sampling set of minimal path density is ill-posed:  the most that one can say in general about a $(d-1)$-dimensional mobile sampling surface in $\R^d$ is that its lower path density is  positive. 

As far as we are aware, all the available sufficiency results for mobile sampling concern surfaces very regularly distributed in space, with precise results concerning special classes of sets with one-dimensional features, such as parallel lines, concentric circles, and spiral sets \cite{UV12, UV13, NJR,RUZ}, and the more general results of Strichartz \cite{Str}  and Jaming-Malinnikova \cite{JM} concerning sampling sets for Sobolev and Besov space functions.

In contrast with these results, our goal in this paper is to provide a general sufficiency condition for mobile sampling \emph{in terms of the lower path density of $\Gamma$ alone}, that is valid for a very large class of surfaces (and even fractal sets), in the spirit of classical one-dimensional sampling results of Beurling and Kahane. A result of this type is only possible in the case of the path density with $k=d-1$. 

We prove that there exists an explicit constant $C_d$ (depending only on the dimension $d$) such that for every closed set $\Gamma\subseteq \R^d$ satisfying a rather mild regularity condition and a Beurling-type density condition $l^-_{d-1}(\Gamma)>C_d \W(K)$ must be a mobile sampling set for $\PW_2(K)$. Here, $\W(K)$ denotes the mean width of the spectrum $K$ (which we assume to be an origin-symmetric convex set). We give the precise statement of our result (Theorem~\ref{mobiledensitythm}) in the following section.  

The aim in this paper is therefore not to find mobile sampling sets of small density, but rather to show that the path density is an appropriate metric insofar as it can provide a guarantee of whether a general surface is mobile sampling.   The result proved here could be useful in circumstances where building a path (or surface) of mobile sensors is more costly in certain spacial locations than others, in which case building a higher concentration of sensors in certain areas of space may be more beneficial than along regularly distributed curves.

\section{Main result}  


For a non-negative integer $k$, put $\omega_{k} = \frac{\pi^{k/2}}{\Gamma(k/2+1)}$ to be the volume of the $k$-dimensional ball in $\R^{k}$.  For $E\subset \R^d$ define $$\H^k(E) = \lim_{\delta\to 0^+}\inf\Bigl\{\omega_k\sum_{j}r_j^k: E\subset \bigcup_j B(x_j, r_j)\text{ and }r_j\leq \delta\Bigl\}.$$
 When restricted to a $k$-dimensional plane, $\H^k=m_k$, where $m_k$ is the $k$-dimensional Lebesgue measure.

Suppose $K\subset \R^d$ is an origin symmetric compact convex set with $d\geq 2$.  Given $1\leq p\leq \infty$, set $\PW_p(K)$ to be the Paley-Wiener space of functions $f\in L^p(\R^d)$ whose Fourier transform as a tempered distribution is supported in $K$.  As mentioned in the introduction, we are interested in the sets of mobile sampling for $\PW_p(K)$, i.e., sets $\Gamma$ for which there is a constant $c>0$ such that
\begin{equation}\label{mobile}c\Bigl(\int_{\R^d}|f|^p dm_d\Bigl)^{1/p} \!\!\leq \Bigl(\int_{\Gamma}|f|^p \, d\H^{d-1}\Bigl)^{1/p}
\!\!\text{ for every }f\in \PW_p(K).\end{equation}
In the case $p=+\infty$, (\ref{mobile}) reads
$$\sup_{\R^d}|f|\leq C\sup_{\Gamma}|f|\text{ for all }f\in \PW_{\infty}(K).
$$
The space $\PW_{\infty}(K)$, consisting of bounded functions whose distributional Fourier transform is supported in $K$, is often referred to as the Bernstein space  \cite{OU}.

\subsection{The surface density} Define the (lower) surface density of a set $\Gamma\subseteq \R^d$ by
\[\D^-(\Gamma)= \liminf_{r\to \infty}\inf_{x\in \R^d}\frac{\H^{d-1}(\Gamma\cap B(x,r))}{m_d(B(x,r))}.\]
Note that our density $\D^-(\Gamma)$ coincides with the path density $l^-_k(\Gamma)$ of Unnikrishnan and Vetterli when $k=d-1$. Since we don't restrict to the case $d=2$, we prefer the terminology of surface density.

\subsection{Regular sets (and measures)} We now introduce the regularity assumption on the set $\Gamma$ that will be assumed in our density result.

 Let $\varphi:[0,1)\to [0,\infty)$ be a function continuous at $0$. A (locally finite Borel) measure $\mu$ is called $\varphi$-regular if for every $x\in \R^d$ and $r\in (0,1)$, $$\mu(B(x,r))\leq \varphi(r) \omega_{d-1}r^{d-1}.$$

A closed set $E\subset \R^d$ is called $\varphi$-regular if the measure $\H^{d-1}|_E$ is $\varphi$-regular. When $d=1$, a set $\Gamma\subset \R$ is uniformly discrete if and only if $\Gamma$ if $\varphi$-regular for a function $\varphi$ with $\varphi(0)=1$.

\subsection{The mean width}  Finally we introduce the quantity through which the spectrum $K$ enters into our density result. The \emph{mean width} of an origin-symmetric convex set $K$ is defined by
$$\W(K) = \frac{2}{\H^{d-1}(\S^{d-1})}\int_{\S^{d-1}}h_K(\theta)\, d\H^{d-1}(\theta),
$$
where $h_K(\theta) = \max_{x\in K} [x\cdot \theta]$ is the support function.  

Geometrically, $2h_K(\theta)$ is the diameter of the orthogonal projection of $K$  onto the line through the origin with direction $\theta$, or alternatively the distance between the two closest supporting hyperplanes to $K$ that are normal to $\theta$. For example,  if $K\subseteq \R^d$ is an origin-centered ball with radius $R$ then $\W(K)=2R$, while if $K$ is an origin-centered cube of side length $R$ then $\W(K)=\frac{2R\omega_{d-1}}{\omega_d}$ (see Section~\ref{example}).

\begin{thm}\label{mobiledensitythm} Set
$$A_d = \frac{\omega_d}{\omega_{d-1}}\cdot\frac{3d^2}{(2d+4)}.
$$
If $\Gamma$ is $\varphi$-regular, and $$\D^-(\Gamma)> \varphi(0)\cdot A_d \cdot \W(K),$$ then $\Gamma$ is a set of mobile sampling for $\PW_p(K)$ for every $1\leq p\leq\infty$, i.e., for every $1\leq p\leq \infty$ there exists a constant $c>0$ such that  (\ref{mobile}) holds.
\end{thm}

We make several remarks about this result:

\begin{enumerate}
\item Properties of the Gamma function ensure that $A_d = O(\sqrt{d})$.  In Section \ref{example} we provide an example to show that one must have $A_d\geq  \frac{d\omega_d}{2\omega_{d-1}}$ and so, if $d=2$, then the constant $A_d$ is within a factor of $3/2$ of being optimal, and has the correct asymptotic dependence on the dimension $d$.
\item Specializing to the case when $d=1$ and $\varphi$ satisfying $\varphi(0)=1$, the theorem states that if $\Gamma$ is uniformly discrete, and $D(\Gamma)>\text{diam}(K)/2$, then (\ref{mobile}) holds.  This is the aforementioned theorem of Beurling and Kahane.
\item  Preiss \cite{Pr} proved that a set $\Gamma$ is (countably) rectifiable if (and only if) the density $\lim_{r\to 0}\frac{\H^{d-1}(B(x,r)\cap \Gamma)}{\omega_{d-1}r^{d-1}}=1$ for $\H^{d-1}$-almost every $x\in \Gamma$.   Consequently, in view of the previous remark, the condition that a set $\Gamma$ is $\varphi$-regular with $\varphi(0)=1$ can be considered a quantitative strengthening of rectifiability.  On the other hand, it is not difficult to construct fractal sets that are $\varphi$-regular if $\varphi(0)>1$. 
\item Although the constant $A_d$ appearing in Theorem \ref{mobiledensitythm} may not be completely sharp for $d\geq 2$, the theorem nevertheless demonstrates that one need not require the surface $\Gamma$ to intersect every ball of a certain fixed radius depending on $K$ in order for mobile sampling to hold (in contrast with the sampling results in \cite{Str, JM} concerning more general classes of functions).
\end{enumerate}

As a final remark, observe that if $\Gamma\subset\R^d$ is $\varphi$-regular, and $\H^{d-1}(\Gamma)>0$ (which is a necessary condition for $\D^-(\Gamma)>0$), then necessarily $\varphi(0)\geq 1$ (see Lemma \ref{density1} below). Therefore, our theorem is a vacuous statement if $\varphi(0)<1$. 
  
\begin{lem}\label{density1}  Suppose that $\Gamma$ is $\varphi$-regular and $\H^{d-1}(\Gamma)>0$, then $\varphi(0)\geq 1$.
\end{lem}

\begin{proof}  Suppose that $\varphi(0)<1$.  Then there exists $\eps>0$ such that $\varphi(r)\leq 1-\eps$ for every $x\in \R^d$ and $r\in (0,\eps)$.

Fix $R>0$ such that $\H^{d-1}(\Gamma\cap B(0,R))\in (0,\infty)$ (the fact that $\H^{d-1}(\Gamma\cap B(0,R))$  is finite for any $R\in (0,\infty)$ is a consequence of the $\varphi$-regularity).  Choose balls $B(x_j, r_j)$ that cover $\Gamma\cap B(0,R)$ with $r_j\in (0,\eps)$ and
$$\sum_{j=1}^{\infty}\omega_{d-1}r_j^{d-1}<(1+\eps)\H^{d-1}(E\cap B(0,R)).
$$
But now, insofar as $\Gamma$ is $\varphi$-regular and $r_j<\eps$,
\begin{equation}\begin{split}\nonumber\H^{d-1}(\Gamma\cap B(0,R))&\leq \sum_{j=1}^{\infty}\H^{d-1}(\Gamma\cap B(x_j, r_j))\leq (1-\eps)\sum_{j=1}^{\infty}\omega_{d-1}r_j^{d-1}\\
&\leq (1-\eps)(1+\eps)\H^{d-1}(\Gamma\cap B(0, R)),
\end{split}\end{equation}
which is absurd.
\end{proof}


There are two main components of the proof of Theorem \ref{mobiledensitythm}. The first component is Proposition \ref{DFprop}, a Ronkin type estimate for the averaged surface area of the nodal set of a function in $\B(K)$.  The second component is a compactness argument leading to a fuzzy variant of the Ronkin estimate (Proposition \ref{compactprop}).  The natural issue that arises in the compactness argument is the lack of good upper semi-continuity properties for the $\H^{d-1}$ measure under (a local variant of) Hausdorff convergence of sets.  We circumvent this issue by employing a relaxation of the problem to measures, and using the $\varphi$-regularity property.  The compactness argument bares some similarities to those quite commonly used in the geometric measure theory, see, e.g. \cite{JTV}.

\section{Proof of Theorem \ref{mobiledensitythm}} 


Fix once and for all a function $\varphi:[0,1)\to (0,\infty)$ that is continuous at $0$.

\subsection{The main propositions} The main estimates concern the space $\PW_{\infty}(K)$.

Our primary function theoretic tool is the following proposition, which is proved in a similar manner to estimates by Ronkin \cite{Ro}.

\begin{prop}\label{DFprop} If $f\in \B(K)$ satisfies $\|f\|_{\infty}\leq 1$ and $|f(0)|>1/2$, then
$$\limsup_{R\to \infty}\frac{1}{\omega_d R^d}\int_0^R\H^{d-1}(B(0,r)\cap \{f=0\})\frac{dr}{r}\leq \frac{A_d}{d}\cdot\W(K),
$$
where as above
$$A_d =\frac{3d^2}{4+2d}\frac{\omega_{d}}{\omega_{d-1}}.
$$
\end{prop}

We were led to prove Proposition \ref{DFprop} from the work of Donnelly and Fefferman \cite[Proposition 6.7]{DF} regarding nodal sets of eigenfunctions.  It was only after proving Proposition \ref{DFprop} that we became aware of the work of Ronkin \cite{Ro} regarding discrete uniqueness sets, which follows a similar path\footnote{More precisely, as in \cite{Ro,DF}, we use Jensen's formula to get a bound on the number of zeroes along any one-dimensional slice, and then integrate over the slices using integral geometry to bound the surface area of the nodal set.}.  Ronkin considers the case when $K$ is a rectangle, and uses different integral geometry than we do here (in particular when generalized to a convex body the estimate would likely not directly involve the mean width).

Proposition \ref{DFprop} will be proved in Section \ref{DFsec}.  In this section we will show how one derives Theorem \ref{mobiledensitythm} from it.  The main goal will be to prove, via a compactness argument, the following ``fuzzy'' version of Proposition \ref{DFprop}:

\begin{prop}\label{compactprop}  Fix $\delta>0$, $R_0>0$.  There exists $\eps>0$ such that for every $\varphi$-regular set $\Gamma$ and $f\in \B(K)$ satisfying $\|f\|_{\infty}\leq 1$ and $|f(x)|>1/2$, there exists $R \geq R_0$ such that
$$\frac{1}{\omega_d R^d}\int_0^R \H^{d-1}(\Gamma\cap B(x,r)\cap \{|f|\leq \eps\})\frac{dr}{r}\leq \varphi(0)(\frac{A_d}{d}\W(K)+\delta).
$$
\end{prop}

\subsection{Compactness preliminaries}    Here we collect the necessary material to execute the compactness argument.  The first lemma is well-known -- see, for instance \cite{OU}.

\begin{lem}\label{BernComp}  Suppose that $f_n\in \B(K)$ satisfy $\|f_n\|_{\infty}\leq 1$.  Then there is a subsequence $f_{n_k}$ that converges uniformly on compact sets to a function $f\in \B(K)$ with $\|f\|_{\infty}\leq 1$.
\end{lem}

The next lemma concerns weak compactness of measures\footnote{Recall that all our measures are non-negative locally finite Borel measures.}.  We say that a sequence of measures $\mu_n$ converges weakly\footnote{Of course, this is an abuse of notation, but it is now standard.} to $\mu$ if 
$$\lim_{n\to \infty}\int_{\R^d}\varphi d\mu_n = \int_{\R^d}\varphi d\mu
$$
for every $\varphi\in C_0(\R^d)$ (the collection of continuous functions $\varphi:\R^d\to \R$ with compact support).

A proof of the following compactness lemma can be found in \cite{Mat}.

\begin{lem}\label{WeakComp}  Suppose that $\mu_n$ is a sequence of measures satisfying 
$$\sup_n\mu_n(B(0,R))<\infty\text{ for every }R>1.
$$
Then there is a subsequence $\mu_{n_k}$ that converges weakly to a measure $\mu$.
\end{lem}

The weak limit satisfies the following lower-semicontinuity properties \cite{Mat}:  Suppose that $\mu_n$ converges to $\mu$ weakly, then
\begin{itemize}
\item $\mu(U)\leq \liminf_{n\to \infty}\mu_n(U)$ for any open set $U\subset \R^n$, and 
\item $\mu(K)\geq \limsup_{n\to \infty}\mu_n(K)$ for any compact set $K\subset \R^n$.
\end{itemize}

\begin{lem}\label{supportlimsup}  Suppose that $\mu_n$ converges weakly to $\mu$, and $x\in \supp(\mu)$.  Then there exists a sequence $\{x_{n_k}\}_k$ with $x_{n_k}\in \supp(\mu_{n_k})$ such that $$\lim_{k\to \infty}x_{n_k}=x.$$
\end{lem}

\begin{proof} Fix $\eps>0$, and choose a function $\varphi\in C_0(\R^d)$ such that $\supp(\varphi)\subset B(0,\eps)$ and $\varphi\equiv 1$ on $B(x,\eps/2)$.  Insofar as $x\in \supp(\mu)$, $0<\int_{\R^d}\varphi d\mu = \lim_{n\to \infty}\int_{\R^d}\varphi d\mu_n\leq \liminf_{n\to \infty}\mu_n(B(x,\eps))$, and therefore for all sufficiently large $n$, there exists $x_n\in \supp(\mu_n)$ with $|x_n-x|<\eps$.
\end{proof}

Let us now specialize weak convergence to $\varphi$-regular sets.

\begin{lem}\label{phireglimit}  Suppose that $\mu_n$ is a sequence of measures that are $\varphi$-regular.  Then there is a subsequence $\mu_{n_k}$ that converges weakly to a $\varphi$-regular measure.
\end{lem}

\begin{proof}  For any $n\in \N$ and $R>0$, the ball $B(0,R)$ can be covered by $CR^d$ balls of radius $1/2$.  Using $\varphi$-regularity we therefore conclude that
$$\sup_{n}\mu_n(B(0,R))\leq CR^d\cdot\varphi(1/2),
$$
and we may apply Lemma \ref{WeakComp} to find a subsequence $\mu_{n_k}$ that converges weakly to a measure $\mu$.  Now fix $x\in \R^d$ and $r\in (0,1)$.  
The lower semi-continuity of the weak limit ensures that
\begin{equation}\begin{split}\nonumber\mu(B(x,r))&\leq \liminf_{k\to \infty} \mu_{n_k}(B(x,r))\leq \varphi(r)\omega_{d-1}r^{d-1}\end{split}\end{equation}
and the lemma follows.\end{proof}

\subsection{The proof of Proposition \ref{compactprop}}  We are now in a position to prove Proposition \ref{compactprop}

\begin{proof}[Proof of Proposition \ref{compactprop}]  Suppose the statement fails to hold.  Then for every $n\in \N$ and a sequence $\eps_n \to 0$ as $n\to \infty$, there exists a $\varphi$-regular set $\wt\Gamma_n$, a function $\wt f_n\in \B(K)$, and $x_n\in \R^d$ satisfying $|\wt f(x_n)|>1/2$, such that for all $R\geq R_0$,
$$\frac{1}{\omega_d R^d}\int_0^R \H^{d-1}(\wt\Gamma_n\cap B(x_n,r)\cap \{|\wt f_n|\leq \eps_n\})\frac{dr}{r}> \varphi(0)(\frac{A_d}{d}\W(K)+\delta).
$$
Put $\Gamma_n = \wt\Gamma_n-x_n$ and $f_n = \wt{f_n}(\,\cdot+x_n)$, so that $\Gamma_n$ is $\varphi$-regular, and $f_n\in \B(K)$ satisfies $\|f_n\|_{\infty}=1$ and $|f_n(0)|>1/2$, but also
$$\frac{1}{\omega_d R^d}\int_0^R \H^{d-1}(\Gamma_n\cap B(0,r)\cap \{| f_n|\leq \eps_n\})\frac{dr}{r}> \varphi(0)(\frac{A_d}{d}\W(K)+\delta)
$$
whenever $R\geq R_0$.

Employing Lemma \ref{BernComp}, by passing to a subsequence we may assume that $f_n$ converge uniformly on compact sets to a function $f\in \B(K)$ (and hence $\|f\|_{\infty}\leq 1$ and $|f(0)|\geq 1/2$).  Additionally, since the measures $\mu_n = \H^{d-1}_{|\Gamma_n\cap \{|f_n|\leq \eps_n\}}$ are $\varphi$-regular, Lemma \ref{phireglimit} ensures that by passing to a subsequence if necessary we may assume that $\mu_n$ converge weakly to a $\varphi$-regular measure $\mu$.  Put $\Gamma = \supp(\mu)$.

From the definition of the Hausdorff measure we infer that for any Borel set $E\subset \R^d$, $$\mu(E)\leq \varphi(0)\H^{d-1}(E).$$ Indeed, if $\delta\in (0,1)$ and $B(x_j, r_j)$ is a cover of $E$ by balls with radius $r_j\leq \delta$, then 
$$\mu(E)\leq \sum_{j}\mu(B(x_j, r_j))\leq \Bigl\{\sup_{r\in (0,\delta)}\varphi(r)\Bigl\}\sum_{j}\omega_{d-1}r_j^{d-1}.
$$
Taking the infimum over such covers of $E$, and then letting $\delta\to 0$ yields the required estimate.

Now, the upper-semicontinuity of the weak limit ensures that, for any $r>0$,

\begin{equation}\begin{split}\nonumber \varphi(0)\H^{d-1}(\Gamma\cap \overline{B(0,r)})&\geq \mu(\overline{B(0,r)})\\&\geq \limsup_{n\to \infty}\H^{d-1}(\Gamma_n\cap \{|f_n|\leq \eps_n\}\cap \overline{B(0,r)}).\end{split}\end{equation}

Insofar as the sets $\Gamma_n$ are $\varphi$-regular, for every $R>0$, the function $r\mapsto \sup_n \H^{d-1}(\Gamma_n\cap \overline{B(0,r)})$ is integrable over $r\in [0,R]$ with respect to the measure $\frac{dr}{r}$, and so the ($\limsup$ variant of the) Fatou Lemma ensures that for $R\geq R_0$
\begin{equation}\begin{split}\nonumber
\int_0^R\H^{d-1}&(\Gamma\cap B(0,r))\frac{dr}{r}=\int_0^R\H^{d-1}(\Gamma\cap \overline{B(0,r)})\frac{dr}{r}\\
&\geq \frac{1}{\varphi(0)}\limsup_{n\to \infty}\int_0^R \H^{d-1}(\Gamma_n\cap\{|f_n|\leq \eps_n\}\cap B(0,r))\frac{dr}{r}\\
&>(\frac{A_d}{d}\W(K)+\delta).
\end{split}\end{equation}

Now fix $x\in \Gamma = \supp(\mu)$. Lemma \ref{supportlimsup} ensures that, by passing to another subsequence if necessary, there is a sequence $x_n\in \supp(\mu_n)$ such that $x_n\to x$.  However, since $f_n \to f$ uniformly on compact sets,  $$|f(x)| = \lim_{n\to \infty}|f_n(x_n)| \leq \lim_{n\to \infty} \eps_n=0,$$ and therefore $\Gamma\subset \{f=0\}$.

We have therefore proved that there exists $f\in \B(K)$ with $\|f\|_{\infty}\leq 1$, $|f(0)|>1/2$ but, for every $R\geq R_0$,
$$\frac{1}{\omega_d R^d}\int_0^R \H^{d-1}(\{f=0\}\cap B(0,r))\frac{dr}{r}>\frac{A_d}{d}\W(K)+\delta.
$$
Given Proposition \ref{DFprop} this is absurd.
\end{proof}


\subsection{The proof of Theorem \ref{mobiledensitythm}} 


\begin{proof}[Proof of Theorem \ref{mobiledensitythm}] Suppose $D(\Gamma)>\varphi(0)A_d\W(K)$.  First fix $\delta>0$ small enough to ensure that \begin{equation}\label{densbigdelta}\D^-(\Gamma)>\varphi(0)[A_d\W(K)+3d\delta].\end{equation}
Consequently, we may fix $R_0$ and a constant $c_0=c_0(\varphi,\delta,\W(K))>0$ such that for all $R\geq R_0$,
\begin{equation}\label{lowdensbdgamma}\frac{1}{\omega_dR^d}\int_{c_0R}^R\H^{d-1}(\Gamma\cap B(0,r))\frac{dr}{r}>\varphi(0)\Bigl[\frac{A_d}{d}\W(K)+2\delta\Bigl].
\end{equation}

Fix $\eps>0$ as in Proposition \ref{compactprop} with these choices of $\delta$ and $R_0$.    Given $f\in \B(K)$ satisfying $\|f\|_{\infty}=1$, choose $x\in \R^n$ such that $|f(x)|>1/2$.  Then by Proposition \ref{compactprop}, there exists $R\geq R_0$ such that
$$\frac{1}{\omega_d R^d}\int_0^R \H^{d-1}(\Gamma\cap B(x,r)\cap \{|f|\leq \eps\})\frac{dr}{r}\leq \varphi(0)\Bigl[\frac{A_d}{d}\W(K)+\delta\Bigl].
$$
Comparing this estimate with (\ref{lowdensbdgamma}) we infer that
$$\frac{1}{\omega_d R^d}\int_{c_0R}^R \H^{d-1}(\Gamma\cap B(x,r)\cap \{|f|> \eps\})\frac{dr}{r}\geq \varphi(0) \delta.
$$
An immediate consequence of this inequality is that $$\sup_{\Gamma}|f|>\eps = \eps\cdot \|f\|_{\infty},$$ and hence Theorem \ref{mobiledensitythm} has been proved in the case $p=+\infty$.  

Now suppose $1\leq p<\infty$, and observe that by the pigeonhole principle, there exists $r\in (c_0R, R)$ such that
$$\H^{d-1}(\Gamma\cap B(x,r)\cap \{|f|> \eps\})\geq c'r^d,
$$
and therefore, as $R\geq c_0R_0$,
$$\|f\|_{\infty}^p\leq C'\frac{1}{r^d}\int_{B(x,r)\cap \Gamma}|f|^p d\H^{d-1}\leq C''\int_{\Gamma}|f|^pd\H^{d-1},
$$
where $C''$ is independent of $f$.
We conclude that there exists $C''$ such that for every $f\in \B(K)$,
\begin{equation}\label{linflp}\|f\|_{\infty}^p\leq C''\int_{\Gamma}|f|^pd\H^{d-1}.
\end{equation}

Continuity of the mean-width ensures that we can choose $\kap>0$ such that, with $K_{\kap} = K+B(0,\kap)$ the $\kap$-neighborhood of $K$, we have
$$D(\Gamma)> \varphi(0)A_d\W(K_{\kap}),
$$
and therefore there is a constant $C>0$ such that the inequality (\ref{linflp}) holds for every $f\in \B(K_{\kap})$.  

Now suppose $f\in \PW_p(K)$.  We follow a standard trick, see for instance \cite[Chapter 2]{OU}.  Choose $h\in \mathcal{S}(\R^d)$ satisfying $h(0)=1$ and $\wh{h}\subset C^{\infty}_0(B(0, \kap))$.  Standard distribution theory ensures that for fixed $x$, the function $g(u) =  f(u)h(u-x)$ lies in $\B(K_{\kap})$.  Therefore, we may apply (\ref{linflp}) to $g$, and hence
\begin{equation}\begin{split}\nonumber\|f\|^p_{L^p(\R^d)}&\leq \int_{\R^d}\sup_{u\in\R^d} |f(u)h(u-x)|^pdm_d(x) \\&\stackrel{(\ref{linflp})}{\leq} C'\int_{\R^d}\int_{\Gamma}|f(u)|^p|h(u-x)|^pd\H^{d-1}(u)dm_d(x)\\&\leq C'\int_{\Gamma}|f|^p d\H^{d-1}.
\end{split}\end{equation}
We  conclude that  (\ref{mobile}) holds.\end{proof}

\section{The proof of the Ronkin estimate}\label{DFsec}

We now return to proving Proposition \ref{DFprop}.  

Recall that $\H^{d-1}(\S^{d-1})= d\cdot \omega_d,$ which (for instance) can be seen via the polar co-ordinates formula 
$\int_{\R^d}f\,dm_d = \int_0^{\infty}\int_{\S^{d-1}}f(r\theta)\,d\H^{d-1}(\theta) r^{d-1}dr.$

 

\begin{lem}\label{logarithmic}  Suppose $f\in \B(K)$ is not identically zero and satisfies $\|f\|_{\infty}\leq 1$.  Then
\begin{equation}\begin{split}
\frac{1}{\omega_dR^d}\int_0^R& \H^{d-1}(\{f=0\}\cap B(0,t))\frac{dt}{t} \leq \W(K) \frac{3d}{4+2d}\frac{\omega_d}{\omega_{d-1}}\\
&+ \frac{(d-1)}{2\omega_dR^{d+2}}\int_{B(0,R)}\log\Bigl(\frac{1}{|f(y)|}\Bigl)\frac{R^2-|y|^2}{|y|}dm_d(y).
\end{split}\end{equation}

\end{lem}

\begin{proof}
It suffices to prove the estimate for $R=1$.  For a general $R>0$, we may consider $f(R\,\cdot\,)$ instead of $f$ (which means replacing $K$ by $RK$), and an elementary change of variable yields the required identity.  So let us henceforth assume that $R=1$.

Fix $\theta\in \S^{d-1}$ and $y\in \theta^{\perp}$.  Set $\ell_{y, \theta}$ to be the line through $y$ with direction $\theta$.  Notice that $y$ is the closest point to $0$ in the line $\ell_{y,\theta}$, so is the mid-point of the line segment $\ell_{y,\theta}\cap \overline{B (0,1)}$.  The function $f_{y,\theta}:\R\to \C$ given  by $f_{y,\theta}(t) = f(\theta\cdot t+y)$ has its one-dimensional Fourier support in the interval $[-h_K(\theta),h_K(\theta)]$.    Therefore $f_{y,\theta}\in \B([-h_K(\theta), h_K(\theta)])$, $\|f_{y,\theta}\|_{\infty}\leq 1$, and hence (for instance, see \cite[Chapter 2]{OU})
\begin{equation}\label{Bgrowth}
|f_{y,\theta}(t+is)|\leq \|f_{y,\theta}\|_{\infty}e^{2\pi h_K(\theta) |s|} \leq e^{2\pi h_K(\theta) |s|}\text{ for }t,s\in \R.\end{equation}
Now, since  $\int_0^{2\pi}|\sin\theta|d\theta=4$,
employing Jensen's formula yields
$$\int_0^r\card(\{t\in [-s,s]: f_{y,\theta}(t)=0\})\frac{ds}{s}\leq 4 h_K(\theta) \cdot r +\log(1/|f_{y, \theta}(0)|). 
$$
Consequently,
\begin{equation}\begin{split}\label{linezeroes}
\int_0^1&\card(\{f=0\}\cap \ell_{y,\theta}\cap \overline{B(0,t)})\frac{dt}{t} \\&=\int_{|y|}^1\card(\{f=0\}\cap \ell_{y,\theta}\cap \overline{B(0,t)})\frac{dt}{t} \\
&= \int_{|y|}^1 \card(\{s\in [-\sqrt{t^2-|y|^2},\sqrt{t^2-|y|^2}]: f_{y,\theta}(s)=0\})\frac{dt}{t}\\
&=\int_0^{\sqrt{1-|y|^2}} \card(\{t\in [-u,u]: f_{y,\theta}(t)=0\})\frac{u^2}{u^2+|y|^2}\frac{du}{u}\\
&\leq 4h_K(\theta) (1-|y|^2)^{3/2}+(1-|y|^2)\log (1/|f(y)|),
\end{split}\end{equation}
where in the third inequality we have employed the substitution $u=\sqrt{t^2-|y|^2}$ and then used that $u\mapsto \frac{u^2}{u^2+|y|^2}$ is non-decreasing in the final inequality.

Since $f\not\equiv 0$ is a real analytic function in $\R^d$, the nodal set $\{f=0\}$ is countably $(d-1)$-rectifiable in the sense of Federer \cite[3.2.14]{F}.  This enables us to be able to utilize the following generalization of the Crofton formula (which can be proved (for instance) as a consequence of the co-area formula).

\begin{fact}\cite[3.2.26]{F}\label{federer}  If $E$ is countably $(d-1)$-rectifiable, 
\begin{equation}\label{Fedeq}\H^{d-1}(E) =\frac{1}{2\omega_{d-1}}\int_{\S^{d-1}}\int_{y\in \theta^{\perp} } \card\bigl(E\cap \ell_{y,\theta}\bigl)dm_{d-1}(y)d\H^{d-1}(\theta),
\end{equation}
where $\ell_{y,\theta}$ is the line through $y$ with angle $\theta$. 
\end{fact}

Let us verify that the constant appearing in the equality (\ref{Fedeq}) is correct:  Consider the case when $E=\S^{d-1}$. For every $\theta\in \S^{d-1}$ and $y\in \theta^{\perp}$, the line $\ell_{y,\theta}$ intersects $\S^{d-1}$ twice if $y\in B(0,1)$ and doesn't intersect $\S^{d-1}$ if $|y|>1$.  Whence 
$$\int_{\S^{d-1}}\int_{y\in \theta^{\perp} } \card\bigl(\S^{d-1}\cap \ell_{y,\theta}\bigl)dm_{d-1}(y)d\H^{d-1}(\theta)= 2\omega_{d-1}\H^{d-1}(\S^{d-1}),$$
as required.\\


We will require one more elementary measure theoretic fact:

\begin{fact}\label{integration}For a non-negative Borel measurable function $g$
\begin{equation}\label{intidentity}\int_{\S^{d-1}}\int_{y\in \theta^{\perp} } g(y) dm_{d-1}(y)d\H^{d-1}(\theta) = (d-1)\omega_{d-1}\int_{\R^d}\frac{g(y)}{|y|}dm_d(y).
\end{equation}
\end{fact}

We give a proof primarily to demonstrate that the factor $(d-1)\omega_{d-1}$ is correct.

\begin{proof}[Proof of Fact \ref{integration}]For a Borel set $E\subset \S^{d-1}$, put $C_E =\bigl\{x\in B(0,1): \frac{x}{|x|}\in E\bigl\}$.  The Borel measure
$$\nu(E) = \int_{\S^{d-1}}m_{d-1}(B(0,1)\cap C_E\cap \theta^{\perp})d\H^{d-1}(\theta)
$$
is a finite rotation invariant Borel measure on $\S^{d-1}$, and so by the uniqueness of such measures, equals $\frac{\nu(\S^{d-1})}{\H^{d-1}(\S^{d-1})}\H^{d-1}|_{\S^{d-1}}.$  Notice that
$$\nu(\S^{d-1}) = \H^{d-1}(\S^{d-1})m_{d-1}(B^{(d-1)}(0,1)),
$$
and thus $\nu = \omega_{d-1}\H^{d-1}|_{\S^{d-1}}.$  Observe now for a sector of the form $rC_E$ for $r>0$, we have that by the homogeneity of the $(d-1)$-dimensional Lebesgue measure,
\begin{equation}\begin{split}\nonumber\int_{\S^{d-1}}& \int_{y\in \theta^{\perp} } \chi_{rC_E} \,dm_{d-1}(y)d\H^{d-1}(\theta)= r^{d-1}\nu(E) \\
&= (d-1)\int_0^rs^{d-2} ds\cdot \nu(E) = (d-1)\omega_{d-1}\int_0^{r}s^{d-2}ds \H^{d-1}(E)\\
&= \omega_d(d-1)\int_{rC_E} \frac{1}{|y|}dm_d(y).
\end{split}\end{equation}
Therefore (\ref{intidentity}) holds when $g=\chi_{rE}$, and therefore also if  $g=\chi_{rC_E\backslash sC_E}$ for $0<s<r<\infty$.  Since any open set can be written as a countable disjoint union of polar rectangles of the form $rC_E\backslash sC_E$ for $E\subset \S^{d-1}$ Borel and $0<s<r<\infty$, Fact \ref{integration} follows.\end{proof}



Let us now apply these measure theoretic facts to our setting with $E=\{f=0\}$.  Averaging Fact \ref{federer} results in
\begin{equation}\begin{split}\nonumber&\int_0^1\H^{d-1}(E\cap B(0,t))\frac{dt}{t} \\&= \frac{1}{2\omega_{d-1}}\int_{\S^{d-1}}\int_{y\in \theta^{\perp} }\int_0^1 \card\bigl(E\cap B(0,t)\cap \ell_{y,\theta}\bigl)\frac{dt}{t}dm_{d-1}(y)d\H_{d-1}(\theta),
\end{split}\end{equation}
into which we plug the inequality (\ref{linezeroes}) and appeal to Fact \ref{integration} to yield
\begin{equation}\begin{split}\nonumber
\int_0^1&\H^{d-1}(\{f=0\}\cap B(0,t))\frac{dt}{t} \\
&\leq \frac{2}{\omega_{d-1}}\int_{\S^{d-1}}h_k(\theta)\int_{B(0,1)\cap \theta^{\perp}}(1-|y|^2)^{3/2}dm_{d-1}(y)dm_{d-1}(\theta) \\&\;\;\;+ \frac{(d-1)}{2}\int_{B(0,1)}\log(1/|f(y)|)\frac{1-|y|^2}{|y|}dm_d(y).
\end{split}\end{equation}


It remains to estimate the first term on the right hand side of this inequality.  For any $\theta\in \S^{d-1}$, and equals 
\begin{equation}\begin{split}\nonumber & \frac{1}{\omega_{d-1}}\int_{B(0,1)\cap \theta^{\perp}}(1-|y|^2)^{3/2}dm_{d-1}(y) = (d-1)\int_0^1(1-r^2)^{3/2}r^{d-1}\frac{dr}{r}\\
&=\frac{(d-1)}{2}\int_0^1(1-s)^{3/2}s^{(d-1)/2}\frac{ds}{s} = \frac{(d-1)}{2}\frac{\Gamma(5/2)\Gamma((d-1)/2)}{\Gamma(d/2+2)} \\&= \frac{\Gamma(5/2)\Gamma((d+1)/2)}{\Gamma(d/2+2)} =\frac{3}{2(2+d)}\frac{\omega_{d}}{\omega_{d-1}}.
\end{split}\end{equation}
Recalling that $\H^{d-1}(\S^{d-1})=d\omega_d$, the lemma is proved.
\end{proof}

Setting $\sigma = \diam(K)$, to complete the proof of Proposition \ref{DFprop}, it suffices to prove the following result.

\begin{lem}\label{loglimit}  Suppose $f\in \B(B(0,\sigma))$ is not identically equal to zero, $\|f\|_{\infty}\leq 1$ and $|f(0)|\geq \frac{1}{2}$.  Then
$$\limsup_{R\to \infty}\frac{1}{\omega_dR^{d}}\int_{B(0,R)}\log\Bigl(\frac{1}{|f(y)|}\Bigl)\frac{1}{|y|}dm_d(y)=0.$$
\end{lem}

\begin{proof}
It suffices to prove the claim for real valued $f\in \B(B(0, \sigma))$ satisfying $|f(0)|>1/4$, since if $f\in \B(B(0,\sigma))$ is complex valued, then, because $B(0,\sigma)$ is origin symmetric, its real and imaginary parts belong to $\B(B(0,\sigma))$, and replacing $f$ by either its real or imaginary part only increases the integral appearing in the lemma.

We make two claims:
\begin{cla}\label{bdd} The (non-negative) function 
$$R\mapsto \sup_{\theta\in \S^{d-1}}\frac{1}{\omega_d R^{2}}\int_0^R \log\Bigl(\frac{1}{|f(r\theta)|}\Bigl)dr
$$
is bounded on $[1,\infty)$.
\end{cla}
\begin{cla}\label{limit}  For each $\theta\in \S^{d-1}$, 
$$\lim_{R\to \infty}\frac{1}{R^{2}}\int_0^R \log\Bigl(\frac{1}{|f(r\theta)|}\Bigl)dr=0
$$
\end{cla}

After expressing the integral appearing on the left hand side of the conclusion of the lemma in polar co-ordinates:
$$\frac{1}{\omega_d R^{d}}\int_{\S^{d-1}}\int_0^R \log\Bigl(\frac{1}{|f(r\theta)|}\Bigl)r^{d-2}dr d\H^{d-1}(\omega),
$$
we see that the lemma follows immediately from these two claims via the dominated convergence theorem.\end{proof}

Let us return to prove these two claims.

\begin{proof}[Proof of Claim \ref{limit}]  This is a classical result.  The function $r\mapsto f(r\theta)$ is in $\B([-h_K(\theta), h_K(\theta)])$, and therefore
$$\int_0^{\infty}\frac{\log\bigl(\frac{1}{|f(r\theta)|}\bigl)}{1+r^2}dr<\infty.
$$
Claim \ref{limit} is now a consequence the Lebesgue dominated convergence theorem.
\end{proof}

For the proof of Claim \ref{bdd} we shall appeal to a Remez inequality.  The following inequality is a well-known simple special case of much more general results, for instance \cite{Br, NSV2}, but we give a concise proof in an appendix for the benefit of the reader.

\begin{lem}\label{simpleremez}  There is a constant $C>0$ such that the following inequality holds:  Fix $\sigma>0$.  Suppose $g\in \B([-\sigma,\sigma])$ is real valued and satisfies $|g(0)|>1/4$, then for any $F\subset [0,R)$ of positive Lebesgue measure
$$\sup_{[0, R)}|g|\leq C\Bigl(\frac{2eR}{m_1(F)}\Bigl)^{C+e\sigma R}\|g\|_{L^{\infty}(F)}.$$
\end{lem}

With this lemma in hand, we can complete the

\begin{proof}[Proof of Claim \ref{bdd}]  By Lemma \ref{simpleremez} with $g=f(\,\cdot\,\theta)$, we can find $C>0$ such that for all $\theta\in \S^{d-1}$,
$$m_1([0,R)\cap \{|f(\,\cdot\,\theta)|<\eps\})\leq CR\cdot (C\eps)^{1/C\sigma R}.$$
But then
\begin{equation}\begin{split}\nonumber\int_0^R\log\Bigl(\frac{1}{|f(r\theta)|}\Bigl)dr &=\int_0^{\infty}m_1([0,R)\cap \{|f(\,\cdot\,\theta)|<e^{-\lambda}\})d\lambda\\
&\leq CR\int_0^{\infty}e^{-c\lambda/\sigma R}d\lambda = C\sigma R^2.
\end{split}\end{equation}
Claim \ref{bdd} follows.
\end{proof}

\section{Example}\label{example}

Here we provide an example to show that the constant $A_d\geq \frac{\omega_d}{2\omega_{d-1}}$.  Consider the function
$$f(x) = \prod_{n=1}^d \frac{\sin(2\pi  x_n)}{x_n},
$$
which is a constant multiple of the Fourier transform of the cube $K=[-1, 1]^d$.  The function $f$ vanishes on a set $\Lambda =\{f=0\}$ with $\D^-(\Lambda) = 2$. By deleting small regions where any of the planes comprising $\Lambda$ intersect, we can find, for every $\eps>0$, a set $\Lambda_{\eps}\subset \Lambda$, which is $\varphi$-regular for a function satisfying $\lim_{r\to 0}\varphi(r)=1$, such that $\D^-(\Lambda_{\eps})>2-\eps$. Thus, from Theorem \ref{mobiledensitythm}, we must have that $A_d\W([-1,1]^d)\geq  2$.

On the other hand,  we claim that $\W(K) = 4\frac{\omega_{d-1}}{\omega_d}$ for $K=[-1,1]^d$.  To see this, observe that
$$\W(K) = \frac{2}{d\omega_d}\int_{\S^{d-1}}\bigl\{|\theta_1|+\cdots+|\theta_d| \bigl\} d\H^{d-1}(\theta) = \frac{2}{\omega_d}\int_{\S^{d-1}}|\theta_1|d\H^{d-1}(\theta).
$$
Integration by parts yields
\begin{equation}\begin{split}\nonumber
\int_{\R^d} |x_1|e^{-x^2/2} dm_d(x)
&= 2\int_{\R^{d-1}}e^{-|x|^2/2}dm_{d-1}(x)\\& = 2(d-1)\omega_{d-1}\int_0^{\infty}r^{d-2}e^{-r^2/2}dr.
\end{split}\end{equation}
But now observe that
\begin{equation}\begin{split}\nonumber
\int_{0}^{\infty} r^de^{-r^2/2} dr& = (d-1)\int_0^{\infty}r^{d-2}e^{-r^2/2}dr,
\end{split}\end{equation}
and therefore polar co-ordinates yields
$$\int_{\S^{d-1}}|\theta_1|d\H^{d-1}(\theta) = \frac{\int_{\R^d} |x_1|e^{-x^2/2} dm_d(x)}{\int_{0}^{\infty} r^de^{-r^2/2} dr} = 2\omega_{d-1},
$$
as required.

\appendix

\section{The proof of Lemma \ref{simpleremez}}

It clearly suffices to prove the lemma if $\sigma R$ is large. Rescaling the function we may set $R=1$ (and assume $\sigma$ is large).  We follow a standard route for proving a (non-sharp) version of the Remez inequality for polynomials, and in particular the exposition on p.11 of Nazarov-Sodin-Volberg \cite{NSV}.  

Fix $n\in \N$ and choose $t_1,\dots, t_{n+1}\in F$ with $t_1<t_2< t_3 <\cdots< t_{n+1}$ and $t_{i+1}-t_i\geq \frac{\H^1(F)}{n}.$  With $Q(s) = \prod_{j=1}^{n+1} (s-t_j)$, Lagrange interpolation yields
$$g(s) = \sum_{j=1}^{n+1}\frac{g(t_j)Q(s)}{Q'(t_j)(s-t_j)}+\frac{g^{(n+1)}(\xi)Q(s)}{(n+1)!}\text{ for some }\xi\in [-1,1].
$$
Therefore, using that $\|Q\|_{L^{\infty}}\leq1$, $\|g^{(k)}\|_{\infty}\leq \sigma^k$ for every $k\in \mathbb{Z}_+$, and $\Bigl\|\frac{Q(t)}{t-t_j}\Bigl\|_{\infty}\leq 1$,
$$\frac{1}{4}\leq \|g\|_{L^{\infty}(0,t)}\leq \sup_{s\in F}|g(s)|\sum_{j=1}^{n+1}\frac{1}{|Q'(t_j)|}+ \frac{\sigma^{n+1}}{(n+1)!}
$$
Simple estimates  (see p. 11 of \cite{NSV}) yield that
$$\sum_{j=1}^{n+1}\frac{1}{|Q'(t_j)|}\leq \Bigl(\frac{2e}{\H^1(F)}\Bigl)^n.
$$
Thus
$$\frac{1}{4}\leq \Bigl(\frac{2e}{\H^1(F)}\Bigl)^n\sup_{F}|g|+\frac{\sigma^{n+1}}{(n+1)!}.
$$
Now put $n=\lfloor\gamma\cdot \sigma\rfloor +1 $, for $\gamma>1$ to be chosen later, then by Stirling's formula (assume that $\sigma$ is large)
$$(\gamma\sigma+1)! \geq c(\gamma\sigma)^{\gamma\sigma+1}\sqrt{\gamma\sigma}e^{-\gamma\sigma},
$$
so
$$\frac{\sigma^{n+1}}{(n+1)!}\leq C\Bigl(\frac{e}{\gamma}\Bigl)^{\gamma\sigma}\frac{1}{\sqrt{\gamma\sigma}}.
$$
Put $\gamma=e$.  Then for $\sigma$ large enough
$$\frac{1}{4}\leq \Bigl(\frac{2e}{\H^1(F)}\Bigl)^{e\sigma}\sup_{F}|g|+\frac{C}{\sqrt{\sigma}}, \text{ and }1\leq 8 \Bigl(\frac{2e}{\H^1(F)}\Bigl)^{e\sigma}\sup_{F}|g|,
$$
the lemma is proved.

\vspace{.5cm}

\noindent\textbf{Acknowledgment.} The authors would like to thank J.-L. Romero for bringing the problem of mobile sampling to their attention.


\begin{thebibliography}{ABCD}

\bibitem[AGR]{AGR} B. Adcock, M. Gataric, and J. L. Romero. \emph{Computing reconstructions from nonuniform Fourier samples: Universality of stability barriers and stable sampling rates.} Appl. Comput. Harmon. Anal., 2017.

 \bibitem[BW]{BW} J. J. Benedetto and H. C. Wu. \emph{Nonuniform sampling and spiral MRI reconstruction.} In Wavelet
Applications in Signal and Image Processing VIII, volume 4119, pages 130--142. International Society
for Optics and Photonics, 2000.

\bibitem[Beu]{Beu} A. Beurling, \emph{A Balayage of Fourier-Stiltjes transforms}.  In: The collected works of Arne Beurling. Vol. 2. 
Harmonic analysis. Edited by L. Carleson, P. Malliavin, J. Neuberger and J. Wermer. Contemporary Mathematicians. Birkh\"{a}use, Boston, 1989. 
 
 \bibitem[Br]{Br} A. Brudnyi \emph{Local Inequalities for Plurisubharmonic Functions} Ann. of Math., \textbf{149}, No. 2 (1999), pp. 511--533 
 
\bibitem[DF]{DF} H. Donnelly and C. Fefferman, \emph{Nodal sets of eigenfunctions on Riemannian manifolds. }
Invent. Math. \textbf{93} (1988), no. 1, 161--183. 

\bibitem[F]{F} H. Federer,  \emph{Geometric measure theory.} Die Grundlehren der mathematischen Wissenschaften, Band 153 Springer-Verlag New York Inc., New York 1969.

\bibitem[GRUV]{GRUV} K. Gr\"{o}chenig, J.-L. Romero, J. Unnikrishnan, and M. Vetterli, \emph{On minimal trajectories for mobile sampling of bandlimited fields.} Appl. Comput. Harmon. Anal. 39 (2015), no. 3, 487--510. 

\bibitem[JM]{JM} P. Jaming and E. Malinnikova \emph{An uncertainty principle and sampling inequalities in Besov spaces.} J. Fourier Anal. Appl. \textbf{22} (2016), no. 4, 768--786

\bibitem[JTV]{JTV} B. Jaye, X. Tolsa, and M. Villa, \emph{A proof of Carleson's $\varepsilon^2$-conjecture}, Preprint.  arXiv:1909.08581

\bibitem[LS]{LS} V.N. Logvinenko and J. F. Sereda, \emph{Equivalent norms in spaces of entire functions of exponential type},
   Teor. Funkcii Funkcional. Anal. i Prilozen., \textbf{20}, (1974),  102--111

\bibitem[Katz]{Katz}V. E. Katsnelson, \emph{Equivalent norms in spaces of entire functions}, Mat. Sb. (N.S.), \textbf{92}, 1972, 134, 34--54


\bibitem[Kah]{Kah} J.-P. Kahane, \emph{Sur les fonctions moyenne-p\'{e}riodiques bornées.} Ann. Inst. Fourier (Grenoble) \textbf{7} (1957), 293--314. 

\bibitem[Lan]{Lan} H. J. Landau, \emph{Necessary density conditions for sampling and interpolation of certain entire functions.} Acta Math. \textbf{117} (1967), 37--52. 

\bibitem[Pan1]{Pan1} B. Paneah, \emph{On certain theorems of Paley-Wiener type}, Soviet Math. Dokl., \textbf{2}, (1961), 533--536.

\bibitem[Pan2]{Pan2} B. Paneah, \emph{Certain inequalities for functions of exponential type and a priori estimates for general differential operators},
Russian Math. Surveys, \textbf{21} (1966), 3, 75--114.

\bibitem[Mat]{Mat} P. Mattila, \emph{Geometry of sets and measures in Euclidean spaces. Fractals and rectifiability}.
    Cambridge Studies in Advanced Mathematics, \textbf{44}. Cambridge University Press, Cambridge, 1995.

\bibitem[MS]{MS} C. Muscalu and W. Schlag, \emph{Classical and Multilinear Harmonic Analysis. Vol. 1}, Cambridge Studies in Advanced Mathematics,
  \textbf{137}, Cambridge University Press, 2013.

\bibitem[NSV]{NSV} F. Nazarov, M. Sodin, A. Volberg, \emph{Lower bounds for quasianalytic functions, I. How to control smooth functions?}  arXiv:math/0208233.

\bibitem[NSV2]{NSV2}  F. Nazarov, M. Sodin, A. Volberg, \emph{Local dimension-free estimates for volumes of sublevel sets of analytic functions.} Isr. J. Math. \textbf{133}, 269--283 (2003).

\bibitem[NJR]{NJR} F. Negreira, Ph. Jaming and J.-L. Romero, \emph{The Nyquist sampling rate for spiraling curves,} to appear in  Applied and Computational Harmonic Analysis. arXiv:1811.01771. 

\bibitem[OU]{OU} A. Olevskii and A. Ulanovskii, \emph{Functions with disconnected spectrum. Sampling, interpolation, translates.} University Lecture Series, 65. American Mathematical Society, Providence, RI, 2016.

\bibitem[Pr]{Pr} D. Preiss, \emph{Geometry of measures in $\mathbb{R}^n$: distribution, rectifiability and densities}. Ann. Math. 125, 1987, 537--643.


\bibitem[RUZ]{RUZ} A. Rashkovskii, A. Ulanovskii, and I. Zlotnikov, \emph{On 2-dimensional mobile sampling},  arXiv:2005.11193.

\bibitem[Ro]{Ro} L. I. Ronkin, \emph{Discrete sets of uniqueness for entire functions of exponential type of several variables}
Sibirsk. Mat. Z. \textbf{19} (1978), no. 1, 142--152. 


\bibitem[Str]{Str} R. S. Strichartz, \emph{Uncertainty principles in harmonic analysis.} J. Funct. Anal. \textbf{84} (1989), no. 1, 97--114. 

\bibitem[UV12]{UV12} J. Unnikrishnan, and M. Vetterli, \emph{Sampling high-dimensional bandlimited fields on low-dimensional manifolds},
IEEE Transactions on Information Theory, \textbf{59} (2012), no. 4, 2013--2127.

\bibitem[UV13]{UV13} J. Unnikrishnan, and M. Vetterli, \emph{Sampling and reconstruction of spatial fields using mobile sensors},
IEEE Transactions on Signal Processing, \textbf{61} (2013), no. 9, 2328--2340.

 
\end{thebibliography}
\end{document}